\documentclass[11pt, reqno]{amsart}
\usepackage{amssymb}
\usepackage{amsfonts}
\usepackage{amsmath}
\usepackage{graphicx}
\usepackage{xcolor}
\usepackage{mathrsfs}
\usepackage{stmaryrd}
\usepackage{epsfig,color}
\usepackage{blindtext}
\usepackage{enumerate}
\usepackage{hyperref}
\usepackage{url}
\usepackage{bbm}
\usepackage{nicefrac,mathtools}
\usepackage{bm}  
\usepackage[alphabetic,msc-links,lite]{amsrefs}

\DeclareGraphicsExtensions{.pdf,.jpeg,.png}
\usepackage{epstopdf}
\usepackage{cancel} 
\usepackage[normalem]{ulem} 
\usepackage{verbatim} 
\usepackage{enumitem} 
\usepackage{tikz-cd}
\usetikzlibrary{cd}
\usepackage{color}
\usepackage{geometry}
\usepackage{tikz}
\usetikzlibrary{decorations.markings}
\usetikzlibrary{arrows.meta}

\usepackage{extarrows}
\newtheorem{theorem}{Theorem}[section]

\newtheorem{proposition}[theorem]{Proposition}
\newtheorem{lemma}[theorem]{Lemma}
\newtheorem{corollary}[theorem]{Corollary}
\newtheorem{question}[theorem]{Question}

\theoremstyle{definition}

\newtheorem{remark}[theorem]{Remark}

\newcommand\QQ{\mathbb{Q}}
\newcommand\CC{\mathbb{C}}
\newcommand\PP{\mathbb{P}}

\newcommand\ZZ{\mathbb{Z}}

\newcommand\FF{\mathbb{F}}
\newcommand\OO{\mathcal{O}}

\DeclareMathOperator{\et}{\acute{e}t}

\DeclareMathOperator{\Spec}{Spec}

\DeclareMathOperator{\End}{End}

\DeclareMathOperator{\Bl}{Bl}

\title{Non-liftable varieties via \'etale cohomology rings}

\author{Runjie Hu}
\address{Department of Mathematics, Texas A\&M University,
College Station, TX 77843, USA}
\email{runjie.hu@tamu.edu}

\author{Siqing Zhang}
\address{Department of Mathematics, Yale University,
New Haven, CT 06511, USA}
\email{siqing.zhang.math@gmail.com}

\date{}
\begin{document}

\begin{abstract}
    We construct a smooth projective variety in positive characteristic whose $\mathbb{Q}_{\ell}$-coefficient \'etale cohomology ring is not the scalar extension of any graded $\QQ$-algebra, providing an example of a new type of obstruction to characteristic zero liftability.
\end{abstract}

\maketitle

\section*{Introduction}

Let $p>0$ be a prime number. Let $\ell\neq p$ be another prime number.
A graded $\QQ_\ell$-algebra $R$ is said to admit a $\QQ$-form if there
exists a graded $\QQ$-algebra $R_{\QQ}$ and a graded-algebra isomorphism $R_{\QQ}\otimes_{\QQ}\QQ_\ell\cong R$.
The main result of this paper is the following:

\begin{theorem}\label{thm:main-intro}
    There exists a simply-connected smooth projective variety $X$ over $\overline{\FF}_p$ whose \'etale cohomology ring $H^*(X,\QQ_{\ell})$ admits no $\QQ$-form.
In particular, the variety $X$ does not admit a lift to characteristic $0$.
\end{theorem}

By a characteristic-zero lifting of a smooth proper variety
$X/\overline{\FF}_p$, we mean a DVR $R$ of mixed characteristic $(0,p)$, with
residue field $\overline{\FF}_p$, and a flat proper $R$-scheme $\mathcal X$, together with an isomorphism
$\mathcal X\times_R\overline{\FF}_p \simeq X$.

The existence of such a variety as in Theorem \ref{thm:main-intro} shows that there is an essential difference between the $\ell$-adic \'etale cohomology rings of varieties in characteristics $p$ and $0$.  This is a genuinely new type of example, because, as we will review below, the construction of the known non-liftable varieties typically builds on some pathologies in characteristic $p$.
Compared to these pathologies, the $\ell$-adic \'etale cohomology rings are considered to behave nicely. 

In \cite{Sullivan-MIT-notes}*{p.~88}, Sullivan observed that the topological obstruction to characteristic zero lifting of a simply-connected smooth proper variety in characteristic $p$ is whether its $\QQ_{\ell}$-homotopy data descends to $\QQ$ for every $\ell\neq p$. To the best of our knowledge, Theorem \ref{thm:main-intro} is the first example realizing Sullivan's topological obstruction.

In \cite[p.~45, paragraph above \S18]{grothendieck2021pursuing}, Grothendieck expected that the homotopy types associated with good schemes should, after suitable completion, behave like finite polyhedra.
Theorem \ref{thm:main-intro} provides a negative answer to the following question regarding Grothendieck's expectation.
\begin{question}
For every $\ell\neq p$, is the $\ell$-completion of the \'etale homotopy type of every smooth proper variety in characteristic $p$ homotopy equivalent to the $\ell$-completion of a finite CW complex?    
\end{question}

\subsection*{Previous work on non-liftability.}\;

Serre gave the first examples of smooth projective varieties that do not
lift to characteristic zero.  His construction uses finite group actions
arising from projective representations that cannot themselves be lifted
\cite{Serre-Example}.

A second family of obstructions exploits the failure of Hodge theory.  Tango's study of Frobenius and the
counterexamples of Raynaud and Szpiro initiated a large circle of
constructions violating Kodaira-type vanishing
\cites{Tango-Frobenius,Raynaud-kodaira,Szpiro-kodaira,
Mukai-kodaira,Haboush-Lauritzen-kodaira,Lauritzen-Rao-kodaira,
Totaro-kodaira,Zdanowicz-arithmetic-rigidity,
Achinger-Zdanowicz-kodaira}.
When combined with additional arguments, many of them are shown to be non-liftable to characteristic 0.

A related cohomological direction exploits Betti numbers incompatible
with characteristic-zero Hodge theory.  Beginning with Hirokado,
Calabi--Yau threefolds with $b_3=0$ provided a basic source of such
examples
\cites{Hirokado-calabi-yau,Schroer-calabi-yau,Schoen-calabi-yau}.
Further non-liftable Calabi--Yau threefolds were obtained using fiber
products, small resolutions, and compactifications of Drinfeld
half-spaces
\cites{Hirokado-Ito-Saito-calabi-yau-I,
Hirokado-Ito-Saito-calabi-yau-II,Cynk-van-Straten-calabi-yau,
Langer-Drinfeld}.

Another strategy is to make a specifically characteristic-$p$ datum,
such as Frobenius or an elliptic fibration, intrinsically recoverable
from the variety, so that every lifting would force that datum to lift
as well
\cites{Achinger-Zdanowicz-more,Partsch-elliptic}.
Moret--Bailly families provide another source of smooth projective
varieties that do not lift formally to characteristic zero
\cite{Rossler-Schroer-more}.

The preceding obstructions use properties that are essentially geometric.
Compared to those ones, the obstructions of Esnault-Srinivas-Stix and van Dobben de Bruyn use the $\ell$-adic \'etale  topology more intrinsically, and are closer to our approach in spirit \cites{Esnault-Srinivas-Stix-fundamental-group, van-Dobben-de-Bruyn-more}.
However, those obstructions are still different than ours, because theirs rely on the \'etale fundamental groups, while ours is simply-connected.

\subsection*{Sketch of the main construction}\;

The main body of the paper is devoted to the proof of Theorem \ref{thm:main-intro}. 
The main fact that we use to deduce contradiction is that division algebras are not isomorphic to matrix algebras.
Here is a rough sketch:

Take a supersingular elliptic curve $E$.
Its rational endomorphism algebra 
 $\mathcal D:=\End(E)\otimes_{\mathbb{Z}}\mathbb{Q}$
is a four-dimensional quaternion division algebra over $\QQ$.  Choose
endomorphisms which generate $\mathcal D$.  Inside
$(E\times E)\times\PP^1$, we consider the two coordinate copies of $E$,
the diagonal, and the graphs of these generators.  
For each graph $\Gamma_\gamma$, choose a finite set
$T_\gamma\subset\PP^1$ such that the sets $T_\gamma$ are pairwise
disjoint and their cardinalities are pairwise distinct.  We then blow up the disjoint union of the subvarieties
$\Gamma_\gamma\times\{t\}$, for $t\in T_\gamma$.

Let $Z$ denote the resulting blow-up.  Its cohomology ring
$H^*_{\et}(Z,\QQ_\ell)$ already admits no $\QQ$-form.  Indeed, we 
show that the ring intrinsically recovers a labelled configuration that encodes the
action of $\mathcal D$ on $H^1_{\et}(E,\QQ_\ell)$.
If $H^*_{\et}(Z,\QQ_\ell)$ admitted a $\QQ$-form, the same configuration would give rise to an action of $\mathcal D$ on a two-dimensional $\QQ$-vector space, contradicting that $\mathcal D$ is a four-dimensional division algebra.

Finally, to get the simply-connected example, we embed $Z$ into a large projective space and then blow up $Z$ to get the $X$ in Theorem \ref{thm:main-intro}.
We show that the cohomology ring of $X$ still remembers the configuration above. 
Note that the nonliftability of $Z$, combined with \cite[Proposition~4.3]{Achinger-Zdanowicz-more}, already entails that $X$ is non-liftable, but we need extra argument for the stronger statement that the cohomology ring of $X$ has no $\QQ$-form.

\subsection*{Inspirations from complex geometry}\;

Although our problem concerns varieties in positive characteristic, it
is inspired by ideas from complex geometry.

Firstly, cohomology rings with different coefficients have been used to
distinguish non-homeomorphic Galois-conjugate complex varieties.
Galois-conjugate varieties have isomorphic $\ell$-adic and complex cohomology
algebras, whereas their $\mathbb{Z}$, $\mathbb{Q}$, or $\mathbb{R}$-cohomology algebras need
not be isomorphic
\cites{Charles-conjugate,Schreieder-conjugate,
Shen-Zhang-character-varieties}.
The examples in \textit{loc. cit.} concern the uniqueness of $\ZZ$- or $\QQ$- forms of a $\QQ_{\ell}$-algebra given by the cohomology of a variety, whereas the example in our paper concerns the existence.

Secondly, in her counterexample to the Kodaira problem, Voisin made endomorphisms of complex tori visible in their
cohomology rings by blowing up coordinate subtori, the diagonal, and
graphs of endomorphisms \cite{Voisin-homotopy-types}. 
Charles developed this mechanism
further in 
\cite{Charles-conjugate}.  We adapt the same idea to encode a generating set
of an entire quaternion endomorphism algebra, see \S\ref{subsec:auxiliary-variety}.

The final idea appears in Schreieder's construction.
There, any isomorphism between the cohomology rings of certain blow-ups induces an isomorphism between the even-degree cohomology
rings of their blow-up centers \cite{Schreieder-conjugate}.
Our final blow-up plays a similar role: it produces a simply connected
variety while retaining the obstruction already encoded in the
cohomology ring of its center, see \S\ref{subsec:simply-connected-variety}.

\subsection*{Organization}\;

 \S\ref{subsec: prep} records some preparatory standard results.  
In \S\ref{subsec:auxiliary-variety}, we construct the auxiliary threefold $Z$ and show that $H^*(Z)$ has no $\QQ$-form.  
In \S\ref{subsec:simply-connected-variety}, we
form $X=\Bl_Z(\PP^9_k)$, which is simply-connected, with cohomology ring that can recover that of $Z$, so the obstruction persists, proving Theorem \ref{thm:main-intro}.

\subsection*{Notation}\;

In this paper, unless otherwise specified, we write $H^*(-)=H^*_{\et}(-,\QQ_\ell)$ and omit Tate twists.

\section{Preparatory results}\label{subsec: prep}

We collect some standard facts to be used in the construction. 

\subsection{$\QQ$-forms and liftings}
\label{subsubsec:Q-forms-and-liftings}
\;

\begin{lemma}
\label{lem:no-Q-form-obstructs-lifting}
Let $X/\overline{\FF}_p$ be smooth and proper.  If
$H^*_{\et}(X,\QQ_\ell)$ admits no $\QQ$-form, then $X$ does not admit a
characteristic-zero lifting.
\end{lemma}

\begin{proof}
Suppose that $\mathcal X$ is a flat proper lifting of $X$ over a
mixed-characteristic DVR $R$, and write $K=\operatorname{Frac}(R)$.
The morphism $\mathcal X\to\Spec R$ is automatically smooth.  Indeed,
its smooth locus contains the special fiber.  The nonsmooth locus is
closed, and its image in $\Spec R$ is closed because $\mathcal X$ is
proper.  Since this image does not contain the closed point of
$\Spec R$, it must be empty.
Let $\overline K$ be an algebraic closure of $K$.  Since $\mathcal X\to\Spec R$ is smooth and proper, we have the specialization isomorphism of cohomology groups
$H^*_{\et}(X,\ZZ/\ell^n\ZZ)
\cong
H^*_{\et}(\mathcal X_{\overline K},\ZZ/\ell^n\ZZ)$;
see \cite[VI, Corollary~4.2]{Milne-etale-cohomology}.
By the description of the specialization isomorphism as in \cite[\href{https://stacks.math.columbia.edu/tag/0GJ3}{Tag 0GJ3}]{stacks-project}, the isomorphism respects the ring structure.
After descending the generic fiber to a finitely generated field over
$\QQ$ and choosing an embedding of that field into $\CC$, let $Y/\CC$
be the resulting complex variety.  Invariance under extension of
algebraically closed fields and the comparison theorem give us that $H^*_{\mathrm{sing}}(Y(\CC),\QQ)$ is a $\QQ$-form of
$H^*_{\et}(X,\QQ_\ell)$; see
\cite[VI, Corollary~4.3 and III, Theorem~3.12]
{Milne-etale-cohomology}.
\end{proof}

\subsection{Cohomology of blow-ups}
\label{subsec:cohomology-of-blowups}\;

We begin by recalling the blow-up formula for \'etale cohomology and its
compatibility with cup products, which will be used in both stages of the
construction.

Let $W$ be a smooth variety and let $i:C\hookrightarrow W$ be a smooth closed subvariety of codimension
$c\geq 2$, and let
\[
f:\widetilde W=\Bl_C(W)\longrightarrow W
\]
be the blow-up.  Denote its exceptional divisor by $j:F\hookrightarrow
\widetilde W$, its projection to $C$ by $g:F\to C$, and set
$\xi=c_1(\OO_F(1))$.

The projective bundle formula and the blow-up formula
\cite{SGA7-II}*{Expos\'e XVIII, formulas (1.2.2) and (2.2.2)}
give an isomorphism
\begin{equation}
    \label{eqn: blowup for}
    H^n(W)\oplus\bigoplus_{r=1}^{c-1}H^{n-2r}(C)
\xrightarrow{\sim} H^n(\widetilde W),
\quad 
(a,\alpha_1,\ldots,\alpha_{c-1})
\mapsto
f^*a+\sum_{r=1}^{c-1}
j_*\bigl(g^*\alpha_r\cdot\xi^{r-1}\bigr).
\end{equation}
Moreover, the projection formula gives
\[
f^*a\cdot j_*\bigl(g^*\alpha\cdot\xi^{r-1}\bigr)
=
j_*\bigl(g^*(i^*a\cdot\alpha)\cdot\xi^{r-1}\bigr).
\]

Therefore, $f^*$ identifies $H^1(W)$ with $H^1(\widetilde W)$. Moreover,
if we write the cohomology class of the exceptional divisor as $e:=j_*(1_F)$ with  $1_F\in H^0(F)$, then we have the isomorphism
\begin{equation}
\label{eqn: proj formula}
    \ker\left(
H^1(\widetilde W)\xrightarrow{\ \cdot e\ }H^3(\widetilde W)
\right)
\cong 
\ker\left(
H^1(W)\xrightarrow{\ i^*\ }H^1(C)
\right).
\end{equation}

Therefore, multiplication by the exceptional class records restriction to the
blow-up center.

\section{The auxiliary variety $Z$}
\label{subsec:auxiliary-variety}

The goal of this section is to construct a smooth projective threefold whose blow-up centers
encode the endomorphism algebra of a supersingular elliptic curve.

\subsection{Construction of $Z$}\;

Set $k=\overline{\FF}_p$, and choose a supersingular elliptic curve $E/k$.
Such a curve exists for every $p$; see
\cite[Chapter~V, \S4]{Silverman-AEC}.  By
\cite[Remark~III.9.2 and Theorem~V.3.1(a)(iv)]{Silverman-AEC}, its
rational endomorphism algebra
\[
\mathcal D:=\End(E)\otimes_{\ZZ}\QQ
\]
is a four-dimensional quaternion division algebra over $\QQ$.

Choose distinct endomorphisms $\alpha,\beta\in\End(E)\setminus\{0,1\}$
which generate $\mathcal D$ as a $\QQ$-algebra, and set
\begin{equation}
    \label{eqn: I}
    I:=\{\infty,0,1,\alpha,\beta\}.
\end{equation}
For $\gamma\in I\setminus\{\infty\}$, let $\Gamma_\gamma\subset E\times E$
be the graph of $\gamma$, and set $\Gamma_\infty:=\{0_E\}\times E$.

The curves $\Gamma_\gamma$ are not disjoint in $E\times E$.  We
separate them using an auxiliary projective line.  For each $\gamma\in I$,
choose a finite nonempty subset $T_\gamma\subset\PP^1(k)$
such that the sets $T_\gamma$ are pairwise disjoint and the integers
$m_\gamma:=\#T_\gamma$ are pairwise distinct.  Set
\[
Y:=(E\times E)\times\PP^1_k
\]
and, for $\gamma\in I$ and $t\in T_\gamma$, let
\[
C_{\gamma,t}:=\Gamma_\gamma\times\{t\}\subset Y.
\]
These curves are pairwise disjoint.  Define
\[
C:=\bigsqcup_{\gamma\in I}\ \bigsqcup_{t\in T_\gamma}C_{\gamma,t},\]
and
\[f:Z:=\Bl_C(Y)\longrightarrow Y.
\]
We have that $Z$ is a smooth projective threefold.

\subsection{Encoding the quaternion algebra in $H^1(Z)$}
\label{subsubsec:encoding-quaternion-algebra}\;

Set $V:=H^1(E)$.
The K\"unneth formula and \eqref{eqn: blowup for} give natural identifications
\[
H^1(Z)\cong H^1(Y)\cong H^1(E\times E)
\cong V\oplus V.
\]

For each $\gamma\in I$ as in \eqref{eqn: I}, define
\[
K_\gamma:=
\ker\left(
H^1(E\times E)\longrightarrow H^1(\Gamma_\gamma)
\right).
\]
For $\gamma\neq\infty$, restriction to the graph of $\gamma$ is given by
\[
V\oplus V\longrightarrow V,
\qquad
(u,v)\longmapsto u+\gamma^*v.
\]
Consequently,
\[
K_\gamma
=
\left\{(-\gamma^*v,v):v\in V\right\}.
\]
For the three distinguished curves, this gives
\[
K_\infty=V\oplus 0,
\qquad
K_0=0\oplus V,
\qquad
K_1=\left\{(-v,v):v\in V\right\}.
\]

After changing the sign on the second copy of $V$, the subspaces
$K_\infty$ and $K_0$ become the two coordinate subspaces, $K_1$ becomes
the diagonal, and $K_\gamma$ becomes the graph of $\gamma^*$.  
Therefore, $K_\infty$, $K_0$, and $K_1$ specify two copies of $V$ together with an
identification between them, while $K_\alpha$ and $K_\beta$ encode the
linear maps $\alpha^*$ and $\beta^*$.  Since $\alpha$ and $\beta$
generate $\mathcal D$, the labeled configuration
\[
(K_\infty,K_0,K_1,K_\alpha,K_\beta)
\]
encodes the action of $\mathcal D^{\mathrm{op}}$ on $V$.  Here the
opposite algebra appears because pullback reverses composition, and this
distinction will not affect the argument.

\begin{lemma}
\label{lem:transverse-subspaces}
For distinct $\gamma,\delta\in I$, we have that $K_\gamma\cap K_\delta=0$.
\end{lemma}

\begin{proof}
Suppose first that $\gamma,\delta\neq\infty$.  A vector in
$K_\gamma\cap K_\delta$ has the form $(-\gamma^*v,v)=(-\delta^*v,v)$
for some $v\in V$.  Therefore, we have that $(\gamma-\delta)^*v=0$.
Since $\gamma-\delta$ is a nonzero element of the division algebra
$\mathcal D$, its action on $V$ is invertible.  Therefore $v=0$.  The
case in which one of the labels is $\infty$ follows immediately from
$K_\infty=V\oplus 0$.
\end{proof}

For $\gamma\in I$ and $t\in T_\gamma$, let $e_{\gamma,t}\in H^2(Z)$ be the cohomology class defined by the exceptional divisor over $C_{\gamma,t}$.  
By \eqref{eqn: proj formula}, we have the identification
\[
\ker\left(
H^1(Z)\xrightarrow{\ \cdot e_{\gamma,t}\ }H^3(Z)
\right)
=
K_\gamma.
\]

Define the exceptional block associated with $\gamma$ by
\[
B_\gamma:=
\bigoplus_{t\in T_\gamma}\QQ_\ell e_{\gamma,t}
\subset H^2(Z).
\]
We have that $\dim_{\QQ_\ell}B_\gamma=m_\gamma$.

In summary, each exceptional block records one subspace $K_\gamma$, while its
dimension records the corresponding label $\gamma\in I$.

\subsection{Recovering the labeled configuration}
\label{subsubsec:recovering-configuration}\;

We now regard $H^*(Z)$ only as an abstract graded algebra.  The purpose of
this section is to prove the following.

\begin{proposition}[Intrinsic recovery]
\label{prop:recovering-configuration}
The graded $\QQ_\ell$-algebra $H^*(Z)$ intrinsically determines the
labeled configuration
\[
(K_\infty,K_0,K_1,K_\alpha,K_\beta)
\]
of two-dimensional subspaces of $H^1(Z)$.

Moreover, let $F\subset\QQ_\ell$ be a subfield, let $R_F$ be a graded $F$-algebra, and suppose that there is an isomorphism $\varphi:R_F\otimes_F\QQ_\ell\xrightarrow{\sim}H^*(Z)$
of graded algebras.  Then there exist two-dimensional $F$-subspaces $K_{\gamma,F}\subset R_F^1$, $\gamma\in I$ such that $\varphi\bigl(K_{\gamma,F}\otimes_F\QQ_\ell\bigr)=K_\gamma$ for every $\gamma\in I$.
\end{proposition}

We need some setup before proving Proposition~\ref{prop:recovering-configuration}.
Set
\[
S:=H^1(Z).
\]
Define the quotient spaces
\[
\overline H^2:=H^2(Z)/S^2,
\qquad
\overline H^3:=H^3(Z)/S^3,
\]
where the product of subspaces denotes the linear space of all their products.

For $\overline{x}\in\overline H^2$, choose a lift
$x\in H^2(Z)$.  Multiplication by $x$ induces a linear map
\begin{equation}
\label{eqn:intrinsic-multiplication-map}
\mu_{\overline{x}}:S\longrightarrow\overline H^3,
\qquad
s\longmapsto s\cdot x\pmod{S^3}.
\end{equation}
This map is independent of the choice of $x$.

The spaces $S^i, i=1,2,3$, $\overline H^2$, and $\overline H^3$, as well as the maps $\mu_{\overline{x}}$, are therefore determined solely by the graded algebra structure of $H^*(Z)$.

Let $\pi_E:Y\to E\times E$, and $\pi_{\PP^1}:Y\to\PP^1$
be the two projections.  We use $f^*\pi_E^*$ to regard
$H^*(E\times E)$ as a subspace of $H^*(Z)$, and set
\[
\theta:=f^*\pi_{\PP^1}^*c_1(\OO_{\PP^1}(1))\in H^2(Z).
\]
We also write $\theta$ for its image in $\overline H^2$.

\begin{lemma}[Low-degree cohomology and multiplication]
\label{lem:low-degree-multiplication}
The following statements hold.

\begin{enumerate}
    \item The pullbacks above give us identifications
    \[
    S^2=H^2(E\times E),
    \qquad
    S^3=H^3(E\times E).
    \]

    \item There are direct-sum decompositions
    \[
    \overline H^2
    =
    \QQ_\ell\theta\oplus
    \bigoplus_{\gamma\in I}B_\gamma,\quad 
    \overline H^3
    =
    S\theta\oplus
    \bigoplus_{\gamma\in I}
    \bigoplus_{t\in T_\gamma}
    H^1(\Gamma_\gamma)e_{\gamma,t}.
    \]

    \item Every $\overline{x}\in\overline H^2$ has a unique expression
    \begin{equation}
    \label{eqn: x bar}
        \overline{x}
    =
    a\theta+
    \sum_{\gamma\in I}
    \sum_{t\in T_\gamma}
    b_{\gamma,t}e_{\gamma,t}, \quad a,b_{\gamma,t}\in \QQ_\ell
    \end{equation}
    For $s\in S$, the multiplication map
    \eqref{eqn:intrinsic-multiplication-map} is given by
    \begin{equation}
    \label{eqn: muxbar}
    \mu_{\overline{x}}(s)
    =
    as\theta+
    \sum_{\gamma\in I}
    \sum_{t\in T_\gamma}
    b_{\gamma,t}
    \bigl(s|_{\Gamma_\gamma}\bigr)e_{\gamma,t}.
    \end{equation}
\end{enumerate}
\end{lemma}

\begin{proof}
The lemma follows from the isomorphism $\bigwedge^2H^1(E)\cong H^2(E)$, the K\"unneth formula, \eqref{eqn: blowup for}, and \eqref{eqn: proj formula}.
\end{proof}

The above formulas enables us to characterize $B_{\gamma}$'s intrinsically as follows. 
Let 
\[\mathscr R:=\left\{
\overline{x}\in\overline H^2:
\dim_{\QQ_\ell}\ker(\mu_{\overline{x}})\geq 2
\right\}.\]
We see that $\mathscr R$ is invariant under the natural scaling action of $\QQ_{\ell}^*$.

\begin{lemma}[The rank-drop locus]
\label{lem:rank-drop-cone}
We have the equality
\begin{equation}
\label{eqn:rank-drop-cone}
\mathscr R
=
\bigcup_{\gamma\in I}B_\gamma.
\end{equation}
Moreover, if $0\neq\overline{x}\in B_\gamma$, then
\begin{equation}
    \label{eqn: kerk}
    \ker(\mu_{\overline{x}})=K_\gamma.
\end{equation}
Consequently, the irreducible components of $\mathscr R$ are given by the linear subspaces $B_\gamma$.
\end{lemma}

\begin{proof}
Let us write $\overline{x}\in \overline{H}^2$ out as in \eqref{eqn: x bar}.
By \eqref{eqn: muxbar}, if $a\neq 0$, then the $S\theta$-component of $\mu_{\overline{x}}(s)$ is $as\theta\neq 0$ for every $0\neq s\in S$, so that $\mu_{\overline{x}}$ is injective.

On the other hand, for any $0\neq x_{\gamma}\in B_{\gamma}$, using \eqref{eqn: muxbar} again, we have that  $\ker(\mu_{x_\gamma})=K_\gamma$.

If $\overline{x}$ has nonzero components in two different blocks, say $B_\gamma$ and $B_\delta$, then 
above combined with Lemma~\ref{lem:transverse-subspaces} entails that
\[
\ker(\mu_{\overline{x}})
\subset K_\gamma\cap K_\delta=0.
\]
The lemma follows.
\end{proof}

\begin{proof}[Proof of Proposition~\ref{prop:recovering-configuration}]
By \eqref{eqn:rank-drop-cone}, we can recover the cone $\bigcup_{\gamma\in I}B_\gamma$ from the algebra structure $H^*(Z)$.
Since $\dim_{\QQ_\ell}B_\gamma=m_\gamma$ are pairwise distinct, we can further recover each $B_{\gamma}$ individually.

Once $B_\gamma$ has been recovered, using \eqref{eqn: kerk}, we can recover 
\[
\bigcap_{\overline{x}\in B_{\gamma}} \ker(\mu_{\overline{x}})=\bigcap_{\overline{x}\in B_{\gamma}\setminus \{0\}} K_{\gamma}=K_{\gamma}.
\]
This proves that the labeled configuration
$(K_\gamma)_{\gamma\in I}$ is determined intrinsically by $H^*(Z)$.

It remains to verify compatibility with extension of coefficients.
Suppose that
\[
\varphi:R_F\otimes_F\QQ_\ell\xrightarrow{\sim}H^*(Z)
\]
is an isomorphism of graded algebras.  Applying the same constructions
over $F$ gives quotient spaces $\overline R_F^2
:=
R_F^2/(R_F^1)^2$, $\overline R_F^3
:=
R_F^3/(R_F^1)^3$,
 the multiplication maps $\mu^F_{\overline{x}}:
R_F^1\longrightarrow\overline R_F^3$,
and a rank-drop cone
\[
\mathscr R_F
:=
\left\{
\overline{x}\in\overline R_F^2:
\dim_F\ker(\mu^F_{\overline{x}})\geq 2
\right\}.
\]
After extending scalars to $\QQ_\ell$, the isomorphism $\varphi$
identifies $\mathscr R_F$ with $\mathscr R$.
The same argument as in Lemma \ref{lem:rank-drop-cone} shows that we can recover the corresponding $K_{\gamma,F}$ such that $\varphi(K_{\gamma,F}\otimes_F\QQ_\ell)=K_{\gamma}$.
\end{proof}

\subsection{The obstruction to a $\QQ$-form}
\label{subsubsec:obstruction-to-Q-form}\;

We now use the recovered configuration to prove that the cohomology ring
of $Z$ cannot be defined over $\QQ$.

\begin{proposition}
\label{prop:Z-has-no-Q-form}
The graded $\QQ_\ell$-algebra $H^*(Z)$ admits no $\QQ$-form.
\end{proposition}

\begin{proof}
Suppose, to the contrary, that there exist a graded $\QQ$-algebra
$R_{\QQ}$ and an isomorphism
$\varphi:
R_{\QQ}\otimes_{\QQ}\QQ_\ell
\xrightarrow{\sim}
H^*(Z)$.
By Proposition~\ref{prop:recovering-configuration}, for each $\gamma\in I$, there is a
two-dimensional $\QQ$-subspace $K_{\gamma,\QQ}\subset R_{\QQ}^1$,
whose extensions to $\QQ_\ell$ are the subspaces $K_\gamma$.
Let $S_{\QQ}:=R_{\QQ}^1$ be the degree 1 piece.
Since $K_\infty$ and $K_0$ are complementary in $H^1(Z)$, the same is
true before extending scalars:
\[
S_{\QQ}
=
K_{\infty,\QQ}\oplus K_{0,\QQ}.
\]
Let $p_\infty:S_{\QQ}\to K_{\infty,\QQ}$ and 
$p_0:S_{\QQ}\to K_{0,\QQ}$
be the corresponding projections.
For $\gamma\in\{1,\alpha,\beta\}$, the restriction $p_0|_{K_{\gamma,\QQ}}:
K_{\gamma,\QQ}\to K_{0,\QQ}$
is an isomorphism.  Therefore, we have that $K_{\gamma,\QQ}$ is the graph of the
$\QQ$-linear map
\[
g_{\gamma,\QQ}
:=
p_\infty\circ
\left(p_0|_{K_{\gamma,\QQ}}\right)^{-1}
:
K_{0,\QQ}\longrightarrow K_{\infty,\QQ}.
\]
Each $g_{\gamma,\QQ}$ is an isomorphism.  
Now we set $V_{\QQ}:=K_{0,\QQ}$
and define
\[
T_{\gamma,\QQ}
:=
g_{1,\QQ}^{-1}\circ g_{\gamma,\QQ}
\in\End_{\QQ}(V_{\QQ}),
\qquad
\gamma\in\{\alpha,\beta\}.
\]

After extending scalars to $\QQ_\ell$, the subspace $K_\gamma$ is the
graph of $-\gamma^*$, while $K_1$ is the graph of $-\operatorname{id}$.
Consequently,
\[
T_{\gamma,\QQ}\otimes_{\QQ}\QQ_\ell
=
\gamma^*
\qquad
\text{for }\gamma\in\{\alpha,\beta\}.
\]

The maps $T_{\alpha,\QQ}$ and $T_{\beta,\QQ}$ therefore satisfy all the
$\QQ$-algebra relations satisfied by $\alpha^*$ and $\beta^*$.  Indeed,
any such relation holds after extending scalars to $\QQ_\ell$, and
extension of scalars from $\QQ$ to $\QQ_\ell$ is faithful.  Since
$\alpha$ and $\beta$ generate $\mathcal D$, we obtain a unital
homomorphism
\[
\rho:
\mathcal D^{\mathrm{op}}
\longrightarrow
\End_{\QQ}(V_{\QQ}).
\]
This is absurd, because $\mathcal D^{\mathrm{op}}$ is a 4-dimensional division algebra over $\QQ$, so that a non-zero orbit must have dimension 4, but $\dim_{\QQ}(V_{\QQ})=2$.
\end{proof}

Proposition \ref{prop:Z-has-no-Q-form} and Lemma \ref{lem:no-Q-form-obstructs-lifting} entail that 

\begin{corollary}
$Z$ does not admit a characteristic 0 lifting.    
\end{corollary}

\section{The simply connected variety $X$}
\label{subsec:simply-connected-variety}

\subsection{The final blow-up}
\label{subsubsec:final-blow-up}\;

By \cite[Chapter~II, \S5.4, Theorem~2.25]{Shafarevich-BAG1},
the smooth projective threefold $Z$ admits a closed immersion into
$\PP^7_k$.  Composing with a linear closed immersion
$\PP^7_k\hookrightarrow\PP^9_k$, fix a closed immersion
$i:Z\hookrightarrow\PP^9_k$, and set
\[
X:=\Bl_Z(\PP^9_k).
\]
Birational invariance of the \'etale fundamental group
\cite[Expos\'e~X, Corollaire~3.4]{sga1} entails that $X$ is simply-connected.
We will show that the abstract graded algebra
$H^*(X,\QQ_\ell)$ admits no $\QQ$-form.

Let $\pi:X\to\PP^9_k$ be the blow-up morphism, let
$j:F\hookrightarrow X$ be the exceptional divisor, and let $g:F\to Z$
be its projection.  We consider the two classes
\[
h:=\pi^*c_1\bigl(\OO_{\PP^9_k}(1)\bigr),
\qquad
e:=j_*(1_F)
\]
in $H^2(X)$, where $1_F\in H^0(F)$ is the unit.  Thus $h$ is the
hyperplane class coming from the ambient projective space, while $e$ is
the class of the exceptional divisor.

For $1\leq a\leq 5$, define
\[
\iota_a:H^r(Z)\longrightarrow H^{r+2a}(X),
\qquad
\iota_a(\alpha)
=
j_*\bigl(g^*\alpha\cdot(j^*e)^{a-1}\bigr).
\]
The normal bundle of $F$ in $X$ is $\OO_F(-1)$.  Hence the
self-intersection formula gives
\begin{equation}
\label{eqn: self-int}
    j^*e
=
c_1(N_{F/X})
=
-c_1(\OO_F(1)),
\end{equation}
and therefore
\[
\iota_a(\alpha)
=
(-1)^{a-1}
j_*\bigl(g^*\alpha\cdot c_1(\OO_F(1))^{a-1}\bigr).
\]
The map on the right, without the factor $(-1)^{a-1}$, is the
$a$-th exceptional summand map in the blow-up isomorphism
\eqref{eqn: blowup for}.  Consequently, each $\iota_a$ is injective.

Since $Z$ has codimension $6$ in $\PP^9_k$, the blow-up formula
\eqref{eqn: blowup for} gives, for every $m$, a direct-sum decomposition
\begin{equation}
\label{eqn:blowup-for-X}
H^m(X)
=
\pi^*H^m(\PP^9_k)
\oplus
\bigoplus_{a=1}^{5}
\iota_a\bigl(H^{m-2a}(Z)\bigr).
\end{equation}

In low degrees, \eqref{eqn:blowup-for-X} gives
\[
H^1(X)=0,\qquad
H^2(X)=\QQ_\ell h\oplus\QQ_\ell e,\qquad
H^3(X)=\iota_1\bigl(H^1(Z)\bigr).
\]
We will first
show that the two lines $\QQ_\ell h$ and $\QQ_\ell e$ can be recognized
from the abstract graded algebra $H^*(X)$.

\subsection{Recovering $H^*(Z)$ from $H^*(X)$}
\label{subsubsec:recovering-Z-from-X}\;

The crucial step is to recognize the lines $\QQ_\ell h$ and
$\QQ_\ell e$ intrinsically.  Once these lines are known, recovering
$H^*(Z)$ is a formal consequence of the blow-up formula.

We begin with two useful formulas:

\begin{lemma}[Multiplication in the exceptional summands]
\label{lem:multiplication-exceptional-summands}
For $\delta\in H^*(\PP^9_k)$ and $\alpha,\beta\in H^*(Z)$, we have
\begin{equation}
\label{eqn:ambient-exceptional-product}
\pi^*\delta\cdot\iota_a(\alpha)
=
\iota_a\bigl(i^*\delta\cdot\alpha\bigr).
\end{equation}
Moreover, whenever $a+b\leq 5$,
\begin{equation}
\label{eqn:exceptional-product}
\iota_a(\alpha)\cdot\iota_b(\beta)
=
\iota_{a+b}(\alpha\cdot\beta).
\end{equation}
In particular, we have that $\iota_a(1_Z)=e^a$ for $1\leq a\leq 5$, where $1_Z\in H^0(Z)$ is the unit.
\end{lemma}

\begin{proof}
The projection formula gives, for every $x\in H^*(F)$,
\[
\pi^*\delta\cdot j_*(x)
=
j_*\bigl(j^*\pi^*\delta\cdot x\bigr).
\]
Since $j^*\pi^*\delta=g^*i^*\delta$, taking
$x=g^*\alpha\cdot(j^*e)^{a-1}$ proves
\eqref{eqn:ambient-exceptional-product}.

For the second identity, the projection formula and the
self-intersection formula give
\[
\begin{aligned}
\iota_a(\alpha)\cdot\iota_b(\beta)
&=
j_*\bigl(g^*\alpha\cdot(j^*e)^{a-1}\bigr)
\cdot
j_*\bigl(g^*\beta\cdot(j^*e)^{b-1}\bigr)
\\
&=
j_*\Bigl(
g^*\alpha\cdot(j^*e)^{a-1}
\cdot
j^*j_*\bigl(g^*\beta\cdot(j^*e)^{b-1}\bigr)
\Bigr)
\\
&=
j_*\Bigl(
g^*\alpha\cdot(j^*e)^{a-1}
\cdot
c_1(N_{F/X})
\cdot
g^*\beta\cdot(j^*e)^{b-1}
\Bigr)
\\
&=
j_*\Bigl(
g^*\alpha\cdot(j^*e)^{a-1}
\cdot
j^*e
\cdot
g^*\beta\cdot(j^*e)^{b-1}
\Bigr)
\\
&=
j_*\bigl(
g^*(\alpha\cdot\beta)\cdot(j^*e)^{a+b-1}
\bigr)
\\
&=
\iota_{a+b}(\alpha\cdot\beta).
\end{aligned}
\]
Here the second equality is the projection formula, the third is the
self-intersection formula, and the fourth uses
\eqref{eqn: self-int}.
\end{proof}

\begin{lemma}[The two distinguished lines]
\label{lem:distinguished-lines}
The lines $\QQ_\ell h$ and $\QQ_\ell e$ are determined intrinsically by
the graded algebra $H^*(X)$.  More precisely,
\begin{equation}
\label{eqn:hyperplane-line}
\QQ_\ell h
=
\left\{
x\in H^2(X):x^3\cdot H^3(X)=0
\right\}.
\end{equation}
Once this line has been recovered, for any nonzero
$\widetilde h\in\QQ_\ell h$,
\begin{equation}
\label{eqn:exceptional-line}
\QQ_\ell e
=
\ker\left(
\widetilde h^{\,4}\cdot-:
H^2(X)\longrightarrow H^{10}(X)
\right).
\end{equation}
Both descriptions are compatible with extension of the coefficient
field.
\end{lemma}
\begin{proof}
By \eqref{eqn:blowup-for-X}, every class in $H^2(X)$ can be written
uniquely as $ah+be$.  We first show that the line $\QQ_\ell h$ is
contained in the right-hand side of \eqref{eqn:hyperplane-line}.  For
$u\in H^1(Z)$, equation \eqref{eqn:ambient-exceptional-product} gives
\[
h^3\cdot\iota_1(u)
=
\iota_1\left(
i^*c_1\bigl(\OO_{\PP^9_k}(1)\bigr)^3\cdot u
\right)
=
0,
\]
because the class inside $\iota_1$ belongs to $H^7(Z)$ and $Z$ is a
threefold.  Since $H^3(X)=\iota_1(H^1(Z))$, it follows that
$h^3\cdot H^3(X)=0$.

Conversely, suppose that $x=ah+be$ with $b\neq0$, and choose
$0\neq u\in H^1(Z)$.  
Expanding $x^3=(ah+be)^3$ and using
\eqref{eqn:exceptional-product}, we obtain
\[
\begin{aligned}
x^3\cdot\iota_1(u)
&=
a^3h^3\cdot\iota_1(u)
+3a^2b\,h^2e\cdot\iota_1(u)
+3ab^2\,he^2\cdot\iota_1(u)
+b^3e^3\cdot\iota_1(u)
\\
&=
a^3h^3\cdot\iota_1(u)
+3a^2b\,h^2\cdot\iota_2(u)
+3ab^2\,h\cdot\iota_3(u)
+b^3\iota_4(u).
\end{aligned}
\]
By \eqref{eqn:ambient-exceptional-product}, multiplication by $h$
preserves the index of an exceptional summand.  Thus the first three
terms belong respectively to the first, second, and third exceptional
summands in \eqref{eqn:blowup-for-X}, whereas the last term belongs to
$\iota_4(H^1(Z))$.  
Since $\iota_4$ is injective,
$b^3\iota_4(u)\neq0$.  Thus $x^3$ does not annihilate $H^3(X)$, proving
\eqref{eqn:hyperplane-line}.

It remains to recover the exceptional line.  Multiplying
$\widetilde h$ by a nonzero scalar does not change the kernel in
\eqref{eqn:exceptional-line}, so we may take $\widetilde h=h$.  Equation
\eqref{eqn:ambient-exceptional-product} gives
\[
h^4e
=
\iota_1\left(
i^*c_1\bigl(\OO_{\PP^9_k}(1)\bigr)^4
\right)
=
0,
\]
because $H^8(Z)=0$.  On the other hand, $h^5\neq0$, since it is the
pullback of the fifth power of the hyperplane class on $\PP^9_k$.
Therefore, we have that $h^4(ah+be)=ah^5$, whose kernel is $\QQ_\ell e$.  This proves
\eqref{eqn:exceptional-line}.

The condition in \eqref{eqn:hyperplane-line} is given by homogeneous
polynomial equations determined by multiplication in $H^*(X)$, and its
reduced common zero locus is the line $\QQ_\ell h$.  Once this line is
known, \eqref{eqn:exceptional-line} describes $\QQ_\ell e$ as the
kernel of a linear map determined by multiplication.  These
constructions are intrinsic to the graded algebra and commute with
extension of the coefficient field.
\end{proof}

\begin{proposition}[Reconstruction of the center]
\label{prop:recovering-Z-from-X}
The graded algebra $H^*(X)$ determines $H^*(Z)$ up to isomorphism, in
a manner compatible with extension of the coefficient field.  In
particular, if $H^*(X)$ admits a $\QQ$-form, then $H^*(Z)$ also admits
a $\QQ$-form.
\end{proposition}

\begin{proof}
For $0\leq m\leq 8$, set
\[
K^m
:=
\ker\left(
h^4\cdot-:
H^m(X)\longrightarrow H^{m+8}(X)
\right).
\]
In this range, multiplication by $h^4$ is injective on the summand
pulled back from $\PP^9_k$, while it annihilates every exceptional
summand because $H^{>6}(Z)=0$ and \eqref{eqn:ambient-exceptional-product}.  
Therefore, we have that
\begin{equation}
\label{eqn:exceptional-kernel}
K^m
=
\bigoplus_{a=1}^{5}
\iota_a\bigl(H^{m-2a}(Z)\bigr).
\end{equation}

For $0\leq r\leq 6$, multiplication by $e$ maps $K^r$ into
$K^{r+2}$, so we may define
\[
C^r:=K^{r+2}/eK^r.
\]
By \eqref{eqn:exceptional-kernel}, we have
\[
K^{r+2}
=
\iota_1\bigl(H^r(Z)\bigr)
\oplus
\bigoplus_{a=2}^{5}
\iota_a\bigl(H^{r+2-2a}(Z)\bigr).
\]
On the other hand, \eqref{eqn:exceptional-product} gives
\[
e\cdot\iota_a(\alpha)=\iota_{a+1}(\alpha)
\]
for every relevant exceptional summand.  It follows that
\[
eK^r
=
\bigoplus_{a=2}^{5}
\iota_a\bigl(H^{r+2-2a}(Z)\bigr).
\]
Therefore, using that $\iota_a$ is injective, we have that
\begin{equation}
\label{eqn:first-exceptional-layer}
H^r(Z)\xrightarrow{\sim}C^r,
\qquad
\alpha\longmapsto[\iota_1(\alpha)]
\end{equation}
is an isomorphism.  Taken over all $0\leq r\leq6$, these maps identify
$H^*(Z)$ and $C^*$ as graded vector spaces.

Similarly, for $0\leq r\leq 6$, set
\[
D^r:=eK^{r+2}/e^2K^r.
\]
Multiplication by $e$ induces an isomorphism
\begin{equation}
\label{eqn:first-to-second-layer}
C^r\xrightarrow{\sim}D^r,
\qquad
[\iota_1(\alpha)]\longmapsto[\iota_2(\alpha)].
\end{equation}
Thus the first two exceptional layers are naturally identified by
multiplication with the exceptional class.

For $r+s\leq 6$, multiplication in $H^*(X)$ induces a well-defined map
\[
C^r\otimes C^s\longrightarrow D^{r+s},
\qquad
[x]\otimes[y]\longmapsto[xy].
\]
Indeed, the product of two classes in the first exceptional layer lies
in the second, while changing either representative changes the product
only by a class in the third exceptional layer.  Composing this map with
the inverse of \eqref{eqn:first-to-second-layer} defines a product
$\star$ on $C^*$.  By \eqref{eqn:exceptional-product},
\[
[\iota_1(\alpha)]\star[\iota_1(\beta)]
=
[\iota_1(\alpha\cdot\beta)].
\]
Consequently, \eqref{eqn:first-exceptional-layer} identifies
$(C^*,\star)$ with the graded algebra $H^*(Z)$.

All the operations used above—taking kernels, images, quotients, and
induced multiplication—are defined intrinsically from $H^*(X)$ and the
two lines recovered in Lemma~\ref{lem:distinguished-lines}, and they
commute with extension of the coefficient field.  The choice of a
nonzero generator of $\QQ_\ell h$ does not affect the construction.
Replacing $e$ by $ce$ rescales $\star$ by $c^{-1}$, but the resulting
graded algebra is isomorphic to the original one via multiplication by
$c$.  Thus the reconstructed algebra is determined by $H^*(X)$ up to
isomorphism.  In particular, applying the same construction to a
$\QQ$-form of $H^*(X)$ would produce a $\QQ$-form of $H^*(Z)$.
\end{proof}

Finally, Propositions~\ref{prop:recovering-Z-from-X} and \ref{prop:Z-has-no-Q-form} together entail the main Theorem \ref{thm:main-intro}.

\begin{remark}
    To get a higher-dimensional analogous unliftable variety, we can embed $Z$ into $\mathbb{P}^N$ for any $N\geq 9$ and take the blow up.
\end{remark}

\subsection*{Acknowledgements}\;

R.~H. is supported by 
NSF Grant 2247322.
S.~Z. is supported by an AMS--Simons travel grant. We are grateful to the comments from John Pardon.

\bibliographystyle{amsalpha}
\bibliography{ref}

\end{document}